\documentclass{article}
\usepackage{graphicx} 
\usepackage[english]{babel}

\usepackage[letterpaper,top=2cm,bottom=2cm,left=2cm,right=2cm,marginparwidth=1.75cm]{geometry}

\usepackage[colorlinks=true, allcolors=blue]{hyperref}

\usepackage{chngcntr}
\usepackage{makecell}
\usepackage{graphicx}
\usepackage{longtable}
\usepackage{enumitem}
\usepackage{lscape}
\graphicspath{ {C:/Users/steve/Documents/4H project} }
\usepackage{amsmath}
\usepackage{amssymb}
\usepackage{bbm}
\usepackage{float}
\usepackage{amsfonts}
\usepackage{amsthm}
\usepackage{listings}
\usepackage{xcolor}
\usepackage{mwe}
\definecolor{codegreen}{rgb}{0,0.6,0}
\definecolor{codegray}{rgb}{0.5,0.5,0.5}
\definecolor{codepurple}{rgb}{0.58,0,0.82}
\definecolor{backcolour}{rgb}{0.95,0.95,0.92}

\lstdefinestyle{mystyle}{
    backgroundcolor=\color{backcolour},   
    commentstyle=\color{codegreen},
    keywordstyle=\color{magenta},
    numberstyle=\tiny\color{codegray},
    stringstyle=\color{codepurple},
    basicstyle=\ttfamily\footnotesize,
    breakatwhitespace=false,         
    breaklines=true,                 
    captionpos=b,                    
    keepspaces=true,                 
    numbers=left,                    
    numbersep=5pt,                  
    showspaces=false,                
    showstringspaces=false,
    showtabs=false,                  
    tabsize=2
}

\usepackage{mathtools}
\usepackage{lscape}
\usepackage{amssymb}
\usepackage{caption}
\usepackage{array}
\usepackage[format=plain,
            labelfont=it,
            textfont=it]{caption}
\usepackage{tikz-cd}
\usepackage{mathdots}
\usetikzlibrary{decorations.markings}
\tikzset{negated/.style={
        decoration={markings,
            mark= at position 0.5 with {
                \node[transform shape] (tempnode) {$\backslash$};
            }
        },
        postaction={decorate}
    }
}
\usepackage{bm}
\usepackage{biblatex}

\newcolumntype{M}[1]{>{\centering\arraybackslash}m{#1}}
\newcolumntype{N}{@{}m{0pt}@{}}

\newtheorem{theorem}{Theorem}[section]
\theoremstyle{plain}
\newtheorem{conjecture}[theorem]{Conjecture}
\newtheorem*{theorem*}{Theorem A}

\newtheorem{lemma}[theorem]{Lemma}

\theoremstyle{definition}

\newtheorem{definition}[theorem]{Definition}

\newcommand{\nth}{\textup{th}}
\newcommand{\N}{\mathbb{N}}

\newcommand{\R}{\mathbb{R}}

\newcommand{\F}{\mathbb{F}}

\newcommand{\Z}{\mathbb{Z}}

\newcommand{\norm}[1]{\left\lVert#1\right\rVert}

\numberwithin{equation}{subsection}
\AtBeginBibliography{\small}
\DeclareFieldFormat[article, inbook]{title}{#1}
\title{Low Discrepancy Digital Kronecker-Van der Corput
Sequences}
\date{}
\author{Steven Robertson\footnote{steven.robertson@manchester.ac.uk}}
\begin{document}
\maketitle
\begin{abstract}
    \noindent The discrepancy of a sequence measures how quickly it approaches a uniform distribution. Given a natural number $d$, any collection of one-dimensional so-called low discrepancy sequences $\left\{S_i:1\le i \le d\right\}$ can be concatenated to create a $d$-dimensional \textit{hybrid sequence} $(S_1,\dots,S_d)$. Since their introduction by Spanier in 1995, many connections between the discrepancy of a hybrid sequence and the discrepancy of its component sequences have been discovered. However, a proof that a hybrid sequence is capable of being low discrepancy has remained elusive. This paper remedies this by providing an explicit connection between Diophantine approximation over function fields and two dimensional low discrepancy hybrid sequences. \\
    
    \noindent Specifically, let $\F_q$ be the finite field of cardinality $q$. It is shown that some real numbered hybrid sequence $\mathbf{H}(\Theta(t),P(t)):=\textbf{H}(\Theta,P)$ built from the digital Kronecker sequence associated to a Laurent series $\Theta(t)\in\F_q((t^{-1}))$ and the digital Van der Corput sequence associated to an irreducible polynomial $P(t)\in\F_q[t]$ meets the above property. More precisely, if $\Theta(t)$ is a counterexample to the so called $t$\textit{-adic Littlewood Conjecture} ($t$-$LC$), then another Laurent series $\Phi(t)\in\F_q((t^{-1}))$ induced from $\Theta(t)$ and $P(t)$ can be constructed so that $\mathbf{H}(\Phi,P)$ is low discrepancy. Such counterexamples to $t$-$LC$ are known over a number of finite fields by, on the one hand, Adiceam, Nesharim and Lunnon, and on the other, by Garrett and the author.
\end{abstract}
\section{Introduction}
\noindent Let $d$ be a positive integer. Throughout this paper, $\mu_d(\mathcal{B})$ refers to the $d$-dimensional Lebesgue measure of a measurable set $\mathcal{B}$ in $\R^d$ and $\#S$ to the cardinality of a finite set $S$. For a sequence $(\mathbf{z}_n)_{n\ge1}$ in $\R^d$, define\begin{equation*}
    \#(\mathcal{B},\mathbf{z},N)=\#\{n\in\N: 1\le n \le N, ~ \mathbf{z}_n \in \mathcal{B}\}.
\end{equation*} 
\noindent Recall that a $d$-dimensional sequence $\mathbf{z}=(\mathbf{z}_n)_{n\ge1}$ is \textit{uniformly distributed} if for every box $\mathcal{B}\in [0,1]^d$\begin{equation}\lim_{N\to\infty}\frac{\#(\mathcal{B},\mathbf{z},N)}{N} = \mu_d(\mathcal{B}).\label{ud}\end{equation}  

\noindent The \textit{discrepancy} of a sequence estimates quantitatively and uniformly the speed of convergence in (\ref{ud}).
\begin{definition}\label{discrepancy}
    Let $d$ and $N$ be a positive integers and let $\mathbf{z}=(\mathbf{z}_n)_{n\ge1}$ be a sequence in the unit cube $[0,1)^d$. The \textit{discrepancy} of the finite sequence $(\mathbf{z}_n)_{1\le n \le N}$  is defined as \[D_{N}(\mathbf{z})=\sup_{\mathcal{B}\subset [0,1)^d} \left|\frac{\#(\mathcal{B},\mathbf{z},N)}{N}-\mu_d(\mathcal{B})\right|,\] where the supremum is taken over all axis-parallel boxes $\mathcal{B}\subset [0,1)^d$. Similarly, the \textit{star discrepancy} of a sequence, denoted $D^*_{N}(\mathbf{z})$, is defined with the additional condition that $\mathcal{B}$ must have one corner at the origin. 
\end{definition}
\noindent The following inequality from \cite[Theorem 1.3]{ud} shows that in order to understand how the discrepancy of a sequence behaves up to a multiplicative constant, it is sufficient to study its star discrepancy: \[D^*_{N}(\mathbf{z})\le D_{N}(\mathbf{z})\le 2^d D^*_{N}(\mathbf{z}).\] Famously, Roth \cite{Roth} proved in 1954 that every sequence in the $d$-dimensional unit cube satisfies \begin{equation}D^*_{N}(\mathbf{z})\gg_d \frac{\log^{\frac{d-1}{2}}(N)}{N}\cdot\nonumber\end{equation}
\noindent Above and throughout, $x\gg y$ means $x\ge cy$, where $c$ is some positive constant and $x,y>0$. The complementary symbol $\ll$ is defined similarly. If there is a subscript, then the implied constant is meant to depend on the stated variables only. \\

\noindent Proving if Roth's result is optimal constitutes one of the main open problems in the theory of Uniform Distribution. It is often conjectured\footnote{This conjecture is not widely agreed upon, as many people believe that Roth's result is optimal.} (see \cite{book}, \cite{bad}) that every infinite sequence $\mathbf{z}$ in the $d$-dimensional unit cube satisfies \begin{equation}
    D_{N}(\mathbf{z})\gg_{d}\frac{\log^{d}(N)}{N}\cdot
\label{conj}\end{equation}

\noindent Hence, every infinite sequence $\mathbf{z}$ satisfying \begin{equation}\label{low}
    D_{N}(\mathbf{z})\ll_{d,\mathbf{z}}\frac{\log^{d}(N)}{N}
\end{equation}
\noindent is called a \textit{low discrepancy} sequence. 
For a wider view of low discrepancy sequences, see \cite{survey}. \\

\noindent In a series of recent papers \cite{Levin4,Levin3,Levin2,Levin1}, Levin proved that many of the known low discrepancy sequences also satisfy (\ref{conj}), adding further evidence that (\ref{low}) is, in fact, optimal.\\

\noindent This paper focuses on two examples of low discrepancy sequences - Van der Corput sequences and Kronecker sequences. These are defined properly in Section \ref{sect:LD}, but for the purposes of this introduction only the following properties are necessary: \begin{enumerate}
    \item a Van der Corput sequence is defined entirely by a given a natural number $b\ge2$.
    \item a Kronecker sequence is defined entirely by a given irrational number $\alpha\in\R$.
\end{enumerate}
\noindent However, in order for the Kronecker sequence associated to $\alpha\in\R$ to be low discrepancy, $\alpha$ must satisfy a stronger condition than just being irrational. In order to state this extra condition, define the \textit{distance to the nearest integer} $\norm{\alpha}:=\min_{n\in\Z}|n-\alpha|$ and let the \textit{Badly Approximable Numbers} be the set \begin{equation}\textbf{Bad}:=\left\{\alpha\in[0,1) : \inf_{q\in\N}q\cdot\norm{q\alpha}=c(\alpha)>0 \right\}.\label{bad-def}\end{equation} For more information on $\textbf{Bad}$ and the theory of Diophantine Approximation, see \cite{DA_Book} and \cite{DA}. \\

\noindent The Badly Approximable Numbers are related to Kronecker sequences by the following classical result from Niederreiter \cite[Theorem 3.4]{ud}. \begin{theorem}\label{kron}
    Let $\alpha$ be a real number. Then, $\alpha\in \textbf{Bad}$ if and only if its associated Kronecker sequence is low discrepancy.\end{theorem}
\noindent Theorem \ref{kron} shows there is a strong connection between the Badly Approximable Numbers and the discrepancy of Kronecker sequences derived from them. Therefore, it is not unreasonable to predict that if a real number $\alpha$ satisfied a stronger notion of bad approximability, then its associated Kronecker sequence may also have tighter bounds on its discrepancy. To this end, let $p$ be a prime and define the set

 \begin{equation} \textbf{Bad}_p=\left\{\alpha\in \R: \inf_{\substack{q\in\N\backslash\{0\}\\k\ge0}}q\cdot \norm{q\cdot p^k\cdot \alpha}>0\right\}. \label{badp} \end{equation} \noindent That is, $\textbf{Bad}_p$ is the set of all the numbers $\alpha$ such that $p^k\alpha$ is badly approximable for every integer $k\ge0$ and such that the value of $c(p^k\alpha)$ from (\ref{bad-def}) is uniformly bounded below.\\ 

\noindent Note that $\textbf{Bad}_p\subset\textbf{Bad}$, making this a natural way to strengthen the set $\textbf{Bad}$. However, the $p$\textit{-adic Littlewood Conjecture} ($p$-$LC$) by de Mathan and Teulié \cite{padic} states that the set $\textbf{Bad}_p$ is empty for every prime $p$. Hence, if one wishes to create a low discrepancy sequence, $\textbf{Bad}_p$ is unlikely to be useful. To bypass this issue, one moves to the function field set up. \\

\noindent Throughout this paper, let $q$ be a positive power of the prime number $p$ and let $\F_q$ be the finite field with cardinality $q$. Furthermore, let $\F_q[t]$ be the ring of polynomials with coefficients in $\F_q$ and let $\F_q((t^{-1}))$ be the field of formal Laurent series with coefficients in $\F_q$. An element $\Theta(t)\in\F_q((t^{-1}))$ is expanded as \begin{equation}\Theta(t)=\sum_{i=-h}^\infty a_it^{-i}\label{expansion}\end{equation} where $h$ is an integer, $a_i$ are in $\F_q$ and $a_{-h}\neq0$. Additionally, $h$ is called the \textit{degree} of $\Theta(t)$. \\

\noindent If $q$ is not prime, then $q=p^n$ for some integer $n>1$. In this case, the field $\F_q$ is given by $\F_p[x]/p(x)$ where $p(x)\in\F_p[x]$ is an irreducible polynomial of degree $n$. Therefore, each element of $\F_q$ is represented by a unique polynomial $f(x)\in\F_p[x]$ satisfying $\deg(f(x))<n$. Using a different irreducible polynomial of equal degree in place of $p(x)$ results in an isomorphic field, and so there is no need to specify $p(x)$ further. Hence, for the purposes of this paper, consider $p(x)$ to be fixed. The variable $x$ is used to avoid confusion later on, when Laurent series in the variable $t$ will have coefficients in $\F_q$.\\

\noindent To find the function field analogue of the $p$-$LC$, one replaces each real component of the conjecture with its own function field analogue. Here, the analogue of the real numbers (integers, respectively) is $\F_q((t^{-1}))$ ($\F_q[t]$, respectively). Similarly, the set of prime numbers is replaced by the set of irreducible polynomials. Next, one defines the norm of $\Theta(t)\in\F_q((t^{-1}))$ as \[|\Theta(t)|=q^{\deg(\Theta(t))}.\] Finally, using notation from (\ref{expansion}), the \textit{distance to the nearest polynomial} of $\Theta(t)$ is defined as \[|\langle\Theta(t)\rangle|=\left|\sum_{i=1}^\infty a_it^{-i}\right|.\]

\noindent Analogues of both Van der Corput sequences and Kronecker sequences exist over function fields. The entries of the resulting sequences are in $\F_q((t^{-1}))$, but when each term is evaluated\footnote{If $q$ is not prime, then a bijection between $\F_q$ and $\{i\in\N: 0\le i <q\}$ is also needed. More on this in Section 2.} at $t=p$ they become real sequences in the unit interval. The resulting sequences are given the prefix \textit{digital}. For the purposes of this introduction, only the following properties are relevant. \begin{enumerate}
    \item A digital Van der Corput sequence is defined entirely by the choice of $q$ and the choice of a polynomial $B(t)\in\F_q[t]$ such that $\deg(B(t))\ge1$. This sequence is denoted $\left(V_n(B)\right)_{n\ge0}:=\left(V_n(B(t))\right)_{n\ge0}$.
    \item A digital Kronecker sequence is defined entirely by the choice of $q$ and the choice of irrational Laurent series $\Theta(t)\in\F_q((t^{-1}))$. This sequence is denoted $\left(K_n(\Theta)\right):=\left(K_n(\Theta(t))\right)_{n\ge0}$
\end{enumerate} A $d$-dimensional (digital) \textit{Halton} sequence is defined as $d$ different (digital) Van der Corput sequences concatenated together. For a set of polynomials $\mathbf{B}(t):=\{B_i(t)\in\F_q[t]: 1\le i \le d\}$, the $d$-dimensional digital Halton sequence comprised of the digital Van der Corput sequences associated to the $B_i(t)$ is denoted $\left(V_n(\textbf{B}(t))\right)_{n\ge0}$.\\

\noindent The Van der Corput and Kronecker sequences retain their low discrepancy property when translated into their digital counterparts. Specifically, it was proved in \cite{DigHal} that the digital Van der Corput sequence is low discrepancy. Similarly it was shown in \cite[Theorem 4.48]{digkron} that the digital Kronecker sequence associated to $\Theta(t)\in\F_q((t^{-1}))$ is low discrepancy if $\Theta(t)$ is in the function field analogue of $\textbf{Bad}$, which is defined as follows.\\
\begin{equation*}\label{badff}
    \textbf{Bad}(q):=\left\{\Theta(t)\in\F_q((t^{-1})):\inf_{N(t)\in\F_q[t]\backslash\{0\}}N(t)\cdot |\langle N(t)\cdot \Theta(t)\rangle|>0\right\}.
\end{equation*}
\noindent Additionally, given an irreducible polynomial $P(t)\in\F_q[t]$, the function field analogue of $\textbf{Bad}_p$ is given by
    \begin{equation}\label{badpt}\textbf{Bad}(P(t),q):=\left\{\Theta(t)\in\F_q((t^{-1})): q^{-D({\Theta})}:=
    \inf_{\substack{N(t)\in\F_q[t]\backslash\{0\}\\k\ge0}}N(t)\cdot |\langle N(t)\cdot P(t)^k\cdot \Theta(t)\rangle|>0\right\}.
\end{equation}
Here, the value $D({\Theta})\in\N$ is the called the \textit{deficiency} of $\Theta(t)$. \\

\noindent Just as in the real case, de Mathan and Teulié \cite{padic} conjectured that for any choice of finite field $\F_q$ and any choice of irreducible polynomial $P(t)\in\F_q[t]$, the set $\textbf{Bad}(P(t),q)$ is empty. This is known as the $P(t)$\textit{-adic Littlewood Conjecture}. However, it has since been disproved in a recent groundbreaking paper by Adiceam, Nesharim and Lunnon \cite{Faustin}. \\

\noindent Specifically, Adiceam, Nesharim and Lunnon \cite{Faustin} showed that, when $\F_q$ has characteristic 3, the set $\textbf{Bad}(t,q)$ is not empty. This result has since been extended in a paper by the author and Garrett \cite{My_paper2} which proved that for any choice of irreducible polynomial $P(t)\in\F_q[t]$, the set $\textbf{Bad}(P(t),q)$ is non-empty for fields of characteristic 3, 5, 7 and 11. In each of these cases, an element of $\textbf{Bad}(P(t),q)$ is given explicitly. \\

\noindent In the case where $P(t)\neq t$, every known element in $\textbf{Bad}(P(t),q)$ has been induced from an element of $\textbf{Bad}(t,q)$ by applying the following theorem from \cite[Theorem 1.0.3]{My_paper}. \begin{theorem}\label{induce}
    Let $q\in\N$ be a power of a prime and $P(t)\in\F_q[t]$ be an irreducible polynomial. Then for any $\Theta(t)=\sum_{i=1}^\infty a_it^{-i}$ in $\textbf{Bad}(t,q)$, the Laurent series \[\Theta(P(t)):=\sum_{i=1}^\infty a_i P(t)^{-i}\textup{ lies in }\textbf{Bad}(P(t),q).\]
\end{theorem}
\noindent It is unknown if any elements of $\textbf{Bad}(P(t),q)$ exist that are not induced from an element of $\textbf{Bad}(t,q)$.\\

\noindent As in the real set up, it is reasonable to expect the digital Kronecker sequence associated to some $\Theta(t)\in \textbf{Bad}(P(t),q)$ to have good distribution properties. However, a fundamental result by Schmidt \cite{Schmidt} states that every one dimensional real sequence $\textbf{x}=(k_n)_{n\ge0}$ satisfies \[D_{N}(x)\gg \frac{\log(N)}{N}\cdotp\] Therefore, to take advantage of $\Theta(t)$ being in $\textbf{Bad}(P(t),q)$, one must work in at least two dimensions. To this end, the concept of a hybrid sequence is recalled. The main idea is to create new $d$-dimensional low discrepancy sequences by concatenating $d$ one-dimensional low discrepancy sequences. The discrepancy properties of hybrid sequences have been studied extensively (see \cite{hybrid1,hybrid2,bad} for a non-exhaustive list), but so far no low discrepancy hybrid sequences have been found. One such result \cite[Theorem 2]{bad} that is relevant for the present discussion is the following.  \begin{theorem}\label{bad_bound}Let $d\ge1$ be a natural number, let $\Theta(t)\in\F_q((t^{-1}))$ be in $\textbf{Bad}(q)$ and let $\textbf{B}(t)$ be a $d$-dimensional vector whose entries are coprime polynomials in $\F_q[t]$. Then, the $(d+1)$-dimensional hybrid sequence $(\mathbf{H}_n)_{n\ge0}=\left(K_n(\Theta),V_n(\textbf{B})\right)_{n\ge0}$ satisfies \begin{equation}
    D_{N,\mathbf{H}} \ll \frac{\log^{d+1}(N)}{\sqrt{N}}\cdot
\end{equation}\end{theorem}

\noindent In 2022, Levin \cite{Levin5} suggests that any Laurent series in $\textbf{Bad}(t,q)$ could be used to create a low discrepancy two dimensional hybrid sequence. The main result of this paper confirms Levin's conjecture is correct: \begin{theorem}\label{main}
    Let $q$ be a power of a prime and $P(t)\in\F_q[t]$ be an irreducible polynomial. For any $\Theta(t)\in \textbf{Bad}(t,q)$, let $\Phi(t)\in\textbf{Bad}(P(t),q)$ be the Laurent series induced by Theorem \ref{induce}. Then, the hybrid sequence $(\mathbf{H}_n(\Phi,P))_{n\ge0}=\left(K_n(\Phi),V_n(P)\right)_{n\ge0}$ is low discrepancy in dimension 2.
\end{theorem}
\noindent This provides the first known example of a low discrepancy hybrid sequence. Furthermore, the majority of literature on digital sequences is focused on showing that they have the similar properties as their real counterparts. Hence, if the set $\textbf{Bad}_p$ is empty (as it is conjectured to be), Theorem \ref{main} would provide a rare instance of a result in discrepancy theory that is obtainable for digital sequences and that does not have a non-digital analogue.\\

\noindent The next section defines the Kronecker and Van der Corput sequences, as well as their corresponding digital sequences. Section \ref{sect:proof} contains the proof of Theorem \ref{main} and Section \ref{sect:conj} contains some related open problems as well as some further connections between Diophantine approximation and low discrepancy sequences.

\subsubsection*{Acknowledgments}
\noindent The author is grateful to his supervisor Faustin Adiceam for his consistent support, supervision and advice throughout the duration of this project. Additionally, the author thanks Erez Nesharim and Agamemnon Zafeiropoulos for their discussions and careful proof reading following the 2024 conference `Diophantine Approximation, Fractal Geometry and Related Topics' in Paris. Finally, the author acknowledges the financial support of the Heilbronn Institute.

\section{Van der Corput and Kronecker Sequences}\label{sect:LD}
\noindent This section provides basic definitions for the low discrepancy sequences that are needed to prove Theorem \ref{main}. For further details, see \cite{ud}. The following three related definitions are used multiple times in this section, and hence they are included now. Recall that $q\in\N$ is a positive power of a prime $p$.
\begin{definition}\label{eval}
    Define the \textit{Evaluation Function}  \begin{equation}
        E_q:\F_q\to \{n\in\N: 0\le n <q\} ~~~\text{as}~~~ E_q:f(x)\mapsto f(p).
    \end{equation} It is clear that $E$ is a bijection.\end{definition}\noindent For the purposes of proving Theorem \ref{main}, any bijection from $\F_q$ to $\{n\in\N: 0\le n<q\}$ will suffice. Therefore, one does not need to be concerned by which irreducible polynomial is used in the definition of $\F_q$, as the Evaluation Function will always provide a bijection.\begin{definition} Define the \textit{Laurent Series Evaluation Function}  \begin{equation}
        ev_q: \F_q((t^{-1})) \to \R ~~~\text{as}~~~ ev_q\left(\sum_{i=-h}^{\infty}f_i(x)t^{-i}\right)\mapsto \sum_{i=-h}^\infty E_q(f_i(x))q^{-i}.\label{polyeval}
    \end{equation}
    \noindent When the domain is restricted to irrational Laurent series and polynomials, the Laurent Series Evaluation Function is also a bijection.
\end{definition}\begin{definition}\label{ass}
   Let $n=\sum_{i=0}^{\left\lfloor \log_q(n)\right\rfloor} a_iq^i$ be a natural number expanded in base $q$. Define \textit{polynomial in} $\F_q[t]$ \textit{associated to} $n$ as \[n(t):=\sum_{i=0}^{\left\lfloor \log_q(n)\right\rfloor}E_q^{-1}(a_i) t^i.\] 
\end{definition}

\noindent The low discrepancy sequences that were mentioned in the introduction are now defined. \\

\begin{definition}[Van der Corput Sequence] Let $b>1$ be a natural number. The \textbf{Base-$b$ Van Der Corput} sequence, denoted $\left(v_n(b)\right)_{n\ge1}$, is defined for a natural number $n$ expanded in base $b$ as $\sum_{i=0}^\infty n_ib^{i}$ as follows:  \begin{equation}
    v^{(b)}_n=\sum_{i=0}^\infty \frac{n_i}{b^{i+1}}.
\end{equation}
\end{definition}
\noindent Note that the sum is always finite. The higher dimensional generalisation of the Van der Corput sequence is known as the Halton sequence: let $d\ge2$ be an integer and let $\{b_i: 1\le i \le d\}$ be a set of $d$ coprime natural numbers. Let $(v_n(b_i))_{n\ge1}$ be the Van der Corput sequence for each base $b_i$. Then the $d$-dimensional Halton sequence $\left(\mathbf{v}_n(\mathbf{b})\right)_{n\ge1}$ in $\R^d$ is defined as $\left(\mathbf{v}_n(\mathbf{b})\right)_{n\ge1} =\left(v_n(b_1),\dots,v_n(b_d)\right)_{n\ge1}$. It was proven in \cite{Halton} that $\left(\mathbf{v}_n(\mathbf{b})\right)_{n\ge1}$ is a $d$-dimensional low discrepancy sequence.\\

\noindent \begin{definition}[Digital Van der Corput Sequence]
 Let $n$ be a natural number and let $n(t)\in\F_q[t]$ be the polynomial associated to it. For a given polynomial $B(t)\in\F_q[t]$,  expand $n(t)$ in base $B(t)$ as \begin{equation}n(t)=\sum_{i=0}^\infty b_i(t) B(t)^i\label{bi}\end{equation} where $\deg(b_i(t))<\deg(B(t))$. As $n(t)$ has finite degree, $b_i(t)=0$ for all $i>\frac{\deg(n(t))}{\deg(B(t))}$. Then, the \textbf{Digital Van der Corput} sequence associated to $B(t)$ is given by $\left(V_n(B)\right)_{n\ge0}:=\left(V_n(B(t))\right)_{n\ge0}$, where\begin{equation}\label{vandef}V_n(B)=\sum_{i=0}^\infty \frac{ev_q(b_i(t))}{|B(t)|^{i+1}}.\end{equation} 
 \end{definition} \noindent It was proved in \cite{DigHal} that this sequence has low discrepancy.\\

\noindent \begin{definition}[Kronecker Sequences] Let $\alpha\in\R$ be an irrational number and let $\{\alpha\}=\alpha-\lfloor \alpha\rfloor$ denote the fractional part of $\alpha$. Then, the \textbf{Kronecker sequence} is defined as $\left(k_n(\alpha)\right)_{n\ge0}:=\left(\{n\alpha\}\right)_{n\ge1}$.\end{definition}\noindent A classical result of Weyl \cite{Weyl} states that the sequence $\left(k_n(\alpha)\right)_{n\ge0}$ is uniformly distributed for every irrational $\alpha$. However, calculating the discrepancy of $\left(k_n(\alpha)\right)_{n\ge0}$ requires more nuanced information about $\alpha$. As stated in Theorem \ref{kron}, the additional constraint that $\alpha$ is badly approximable is necessary and sufficient to ensure the associated Kronecker sequence is low discrepancy. See \cite[Chapter 2.3]{ud} for more information on the behaviour of the discrepancy of $\left(k_n(\alpha)\right)_{n\ge0}$ for any given real number $\alpha$.\\

\begin{definition}[Digital Kronecker Sequences] For a given $\Theta(t)\in\F_q((t^{-1}))$, one defines function field analogue to the Kronecker sequence as \begin{equation}\left(\widetilde{K}_n(\Theta)\right)_{n\ge0}=\left(\widetilde{K}_n(\Theta(t))(t)\right)_{n\ge0}:=\left(\langle n(t)\cdot\Theta(t) \rangle\right)_{n\ge0}\label{tild}\end{equation} where $\langle\cdot\rangle$ denotes the fractional part\footnote{Recall, the fractional part of a Laurent series is the part only made up of negative powers of $t$.} of the input. Then, the \textbf{digital Kronecker sequence} derived from $\Theta(t)$ is given by \[\left(K_n(\Theta)\right)_{n\ge0}=\left(ev_q\left(\widetilde{K}_n(\Theta)\right)\right)_{n\ge0}.\]\end{definition}\noindent It was proved in \cite[4.48]{ud} that if $\Theta(t)\in\textbf{Bad}(q)$, then this sequence is low discrepancy.

\begin{definition}[Hybrid Sequences] A hybrid sequence in $d$ dimensions is a concatenation of $d$ different one dimensional low discrepancy sequences.\end{definition}\noindent The following hybrid sequence is the focus of this paper: for a given $\Theta(t)\in\F_q((t^{-1}))$ and $B(t)\in\F_q[t]$, define the hybrid sequence $$\mathbf{H}(\Theta,B):=(\mathbf{H}_n(\Theta,B))_{n\ge0}:=\left(K_n(\Theta), V_n(B)\right)_{n\ge0}.$$ Theorem \ref{main} states that if $P(t)\in\F_q[t]$ is an irreducible polynomial and $\Phi(t)\in\mathbf{Bad}(P(t),q)$ has been induced by applying Theorem \ref{induce} to some $\Theta(t)\in\textbf{Bad}(t,q)$, then $\mathbf{H}(\Phi,P)$ is low discrepancy.

\section{Proof that $\mathbf{H(\Phi,P)}$ has Low Discrepancy}\label{sect:proof}
\subsection{Prerequisites}
\noindent This subsection contains three lemmas that are required for the proof of Theorem \ref{main}. The first, from \cite[Page 6]{bad}, concerns properties of the Van der Corput sequence. As neither a proof nor a citation is provided in \cite{bad}, the proof is given below.
\begin{lemma}\label{mod}
    Let $p$ be a prime, and $P(t)$ be an irreducible polynomial in $\F_q[t]$ of degree $d$. Furthermore, let $\left(V_n(P)\right)_{n\ge0}$ be the digital Van der Corput sequence in base $P(t)$ as defined in (\ref{vandef}), and let $l$ and $v$ be natural numbers satisfying $v=\sum_{i=0}^{dl}v_iq^i<q^{dl}$. Then, the following are equivalent: \begin{enumerate}
        \item The $n^\nth$ entry of the digital Van der Corput sequence satisfies \begin{equation}\label{lemm0}V_n(P)\in\left[\frac{v}{q^{d l}},\frac{v+1}{q^{d l}}\right).\end{equation}
        \item Define \[R(t):=\sum_{i=0}^{l-1}\left(\sum_{j=0}^{d-1}E_q^{-1}(v_{d(l-i)+j})t^j\right)P(t)^i.\] Then,
     \begin{equation}\label{lemm1}n(t)\equiv R(t) \mod P(t)^l. \end{equation} 
\end{enumerate}
\end{lemma}
\begin{proof}
    \noindent On the one hand, expand $V_n(P)$ in base $q$ as $$\sum_{i=1}^\infty c_i q^{-i}=\sum_{i=1}^\infty \left(\sum_{j=0}^{d-1}c_{di-j}q^j \right)q^{-di}$$ for $c_i\in\N$ satisfying $0\le c_i<q$. Additionally, expand $n(t)$ as in (\ref{bi}) with $B(t)=P(t)$. From the definition of the digital Van der Corput sequence (see (\ref{vandef}), one has that \begin{equation}\label{lem11}ev_q(b_{i-1}(t))=\sum_{j=0}^{d-1}c_{di-j}q^j.\end{equation} 
    
    \noindent On the other hand, equation (\ref{lemm0}) is equivalent to \[\sum_{i=0}^{dl}v_i q^{-d l+i} \le V^{(P)}_n < \sum_{i=0}^{dl}v_i q^{-d l+i} + q^{-d l},\] which is true if and only if $c_i=v_{d l -i}$.\\
    
    \noindent Hence, equation (\ref{lemm0}) holds if and only if $b_i(t)=\sum_{j=0}^{d-1}E_q^{-1}(v_{d(l-i)+j})t^j$ for $0\le i <l$, and so \[n(t)\equiv R(t):=\sum_{i=0}^{l-1}\left(\sum_{j=0}^{d-1}E_q^{-1}(v_{d(l-i)+j})t^j\right)P(t)^i\mod P(t)^l.\]

\end{proof} 

\noindent The following definition is needed to state the two remaining lemmas.\begin{definition} \label{hank}
A matrix $(h_{i,j})$ for $1\le i\le n, 1\le j \le m$ is called \textbf{Hankel} if all the entries on an anti-diagonal are equal. That is, $h_{i,j}=h_{i+1,j-1}$ for any $i,j\in\mathbb{N}$ such that this entry is defined. Additionally, let $k\in\N$. If $A=(a_i)_{i\in\N}$ is a sequence, then denote the Hankel matrix of size $(m+1)\times(n+1)$ whose entries are $h_{i,j}=a_{i+j+k}$ as $H_A(k,m,n)$
\end{definition} 
\noindent The next lemma, from \cite[Theorem 2.2]{Faustin} states an equivalence between a Laurent series being in $\textbf{Bad}(t,q)$ and a property of Hankel matrices.
\begin{lemma}\label{rank}
    Let $\Theta(t)=\sum_{i=1}^\infty a_i t^{-i}\in\F_q((t^{-1}))$ be a Laurent series whose coefficients are given by the sequence $A=(a_i)_{i\in\N}$. Then, $\Theta(t)\in\textbf{Bad}(t,q)$ with deficiency $D(\Theta)$ if and only if for any positive $k,l\in\N$, the Hankel matrix $H_A(k,l,l+D(\Theta))$ has full rank over $\F_q$. 
\end{lemma}
\noindent Much like Lemma \ref{mod}, the following lemma rephrases what it means for an entry in a digital Kronecker sequence to be in a given interval.
\begin{lemma}\label{kronmat}
    Let $P(t)\in\F_q[t]$ be an irreducible polynomial of degree $d$ and let $\Theta(t)=\sum_{i=0}^\infty a_i P(t)^{-i}\in\F_q((t^{-1}))$ be a Laurent series. Furthermore, let $n,k\in\N$, define $m=\lfloor\log_q(n)\rfloor$ and expand the associated polynomial $n(t)$ in base $P(t)$ as $n(t)=\sum_{i=0}^{\infty}n_i(t)P(t)^i$. Finally, define the sequence $A=(a_i)_{i\ge1}$. Then, the $n^\nth$ term of the digital Kronecker sequence $(K_n(\Theta))_{n\ge0}$ satisfies \begin{equation}\label{kronran}K_n^\Theta\in\left[\frac{k}{q^{dl}}, \frac{k+1}{q^{dl}}\right]\end{equation} if and only if there exists some fixed $l$-dimensional vector $\mathbf{z}$ whose entries are in $\F_q[t]$ with degree less than $d$ such that \begin{equation}
        H_A(1,l-1,m)\left(\begin{array}{c}
            n_0(t)\\\vdots\\n_{m}(t)
        \end{array}\right)=\mathbf{z}.
    \label{matrixl}\end{equation} Above, the precise value of $\textbf{z}$ depends only on $k$.
\end{lemma}
\begin{proof}
    \noindent Recall $\widetilde{K}_n(\Theta)$ from (\ref{tild}). Expanding $K_n(\Theta)$ as \[K_n(\Theta)=\sum_{i=1}^\infty q^{-di}\sum_{j=0}^{d-1}k_{i,j}q^j,\] one has that \[\widetilde{K}_n(\Theta)=\sum_{i=1}^\infty P(t)^{-i}\sum_{j=0}^{d-1}E_q^{-1}(k_{i,j})t^j.\] The range (\ref{kronran}) implies that $k_{i,j}$ is fixed for all $i\le l$ and all $0\le j \le d-1$, as $|P(t)|^{-i}=q^{-di}$. On the other hand, let $n(t)\in\F_q[t]$ be the polynomial associated to $n$. The Laurent series $\widetilde{K}_n(\Theta)$ is defined as \[\widetilde{K}_n(\Theta)=\langle \Theta(t)\cdot n(t) \rangle=\sum_{i=1}^\infty \left(\sum_{j=0}^{m} a_{i+j}n_j(t)\right)P(t)^{-i}.\] Note that this is precisely the matrix product seen in (\ref{matrixl}). Hence, the sum $\sum_{j=0}^ma_{i+j}n_j(t)=\sum_{j=0}^{d-1}E_q^{-1}(k_{i,j})t^j$, concluding the proof.
    
\end{proof}
\subsection{Proof of Theorem \ref{main}}
\begin{proof}
    \noindent Let $d:=\deg(P(t))$ and let $p$ be the prime such that $q$ is a power of $p$. Furthermore, define $M:=\lfloor \log_{q^d}(N)\rfloor$ and let $$\gamma=\sum_{i=1}^\infty \gamma_i q^{-d i}~~~\text{and}~~~\lambda=\sum_{i=1}^\infty \lambda_i q^{-d i}$$ be real numbers in $[0,1)$, with natural numbers $\gamma_i,\lambda_i<q^d$. Define the interval $S:=[0,\gamma)\times[0,\lambda)$ and for any positive natural number $j$ define  $$\Gamma_j:=\sum_{i=1}^j\gamma_iq^{-d i} ~~~\text{and}~~~\Lambda_j:=\sum_{i=1}^j\lambda_iq^{-d i}.$$ For positive natural numbers $j_1,j_2$, define $$I_{j_1,j_2}:=[\Gamma_{j_1},\Gamma_{j_1+1})\times[\Lambda_{j_2},\Lambda_{j_2+1}).$$ The intervals $I_{j_1,j_2}$ are pairwise disjoint and it is clear that \begin{equation}\label{disjoint}S=\bigsqcup_{j_1,j_2\in\N} I_{j_1,j_2}.\end{equation} Clearly, bounding the values of $j_1$ and $j_2$ in (\ref{disjoint}) will produce a set that is strictly contained in $S$. Indeed, define the set \[S':=\bigsqcup_{j_1,j_2\le M}I_{j_1,j_2}\subset S.\]
    \noindent Recall the deficiency $D(\Theta)$ of a Laurent series $\Theta(t)$ from (\ref{badpt}). The union of boxes $S'$ is split into two disjoint parts: \begin{align*}
    S_1&= \bigsqcup_{\substack{j_1,j_2\le M\\ j_1+j_2+2\le M-D(\Theta)}} I_{j_1,j_2}\\
    S_2&= \bigsqcup_{\substack{j_1,j_2\le M\\ j_1+j_2+2> M-D(\Theta)}} I_{j_1,j_2}\end{align*} 
    \noindent The proof relies on covering $S$ in a disjoint union of finitely many boxes. Therefore, small boxes are `glued' on to the edge of $S'$ to create a larger set $S''$ that contains $S$. Specifically, define $S''$ as the union of $S_1$, $S_2$, and the three following disjoint sets \begin{align*}
    S_3 &= \bigsqcup_{j_1<M} [\Gamma_{j_1},\Gamma_{j_1+1})\times\left[\Lambda_{M+1},\Lambda_{M+1}+\lambda_{M+1}q^{-d(M+1)}\right)\\
    S_4 &= \bigsqcup_{j_2<M} \left[\Gamma_{M+1},\Gamma_{M+1}+\gamma_{M+1}q^{-d(M+1)}\right)\times[\Lambda_{j_2},\Lambda_{j_2+1})\\
    S_5 &= \left[\Gamma_{M+1},\Gamma_{M+1}+\gamma_{M+1}q^{-d(M+1)}\right)\times\left[\Lambda_{M+1},\Lambda_{M+1}+\lambda_{M+1}q^{-d(M+1)}\right)
    \end{align*} \begin{figure}[H]
        \centering
        \includegraphics[scale=0.15]{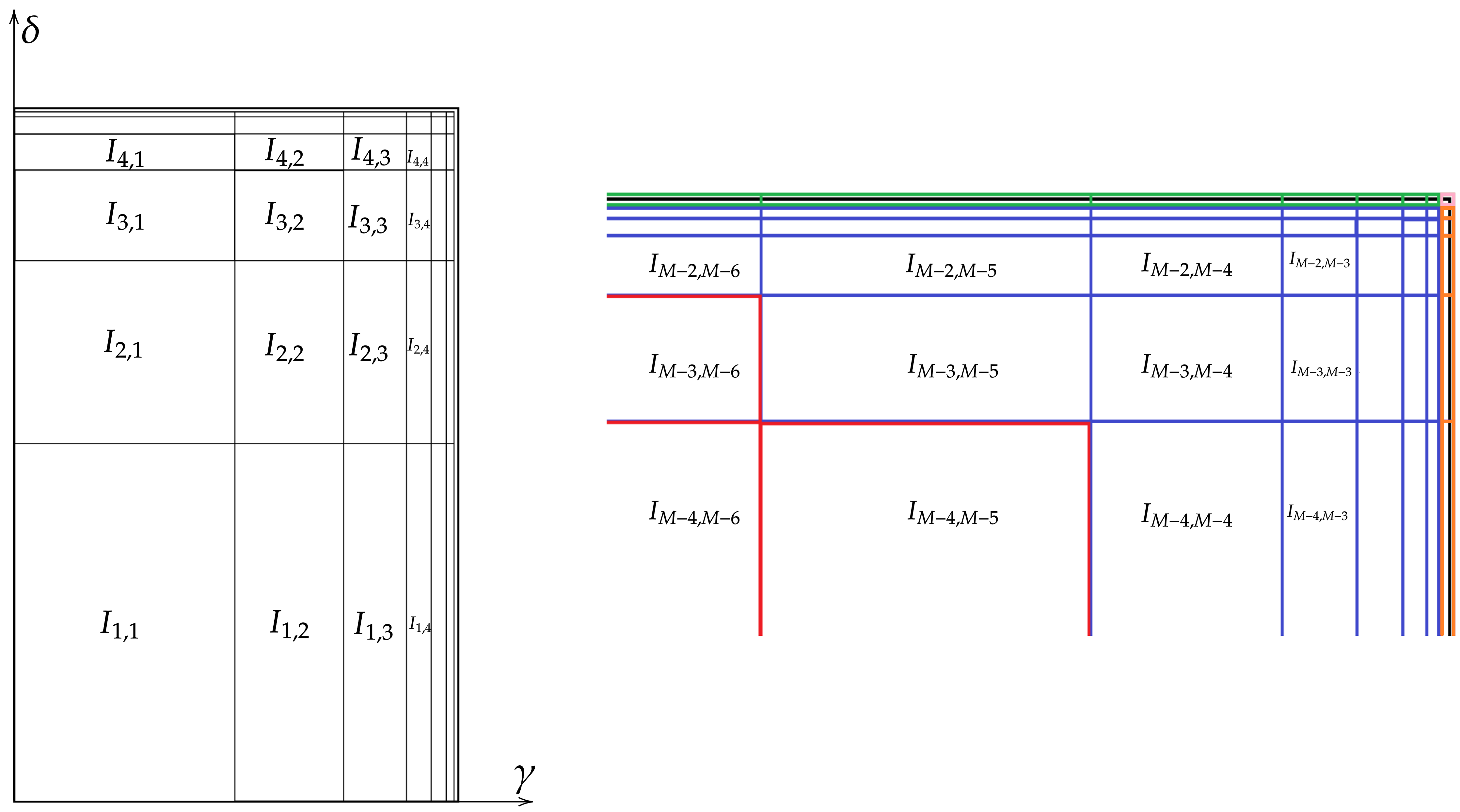}
        \caption{Left: The set $S$ split into boxes. Right: A zoomed in image of the top right corner of $S$, which is drawn in black. The red rectangles refer to the boxes $I_{j_1,j_2}$ in $S_1$, and the blue rectangles are the boxes in $S_2$. The union of the red and blue rectangles is $S'$. The sets $S_3$, $S_4$ and $S_5$ are coloured in green, orange and pink respectively.}
        \label{fig:enter-label}
    \end{figure} 
    \noindent  Any box $I_{j_1,j_2}$ has width $\gamma_{j_1+1}q^{-d(j_1+1)}$ and height $\lambda_{j_2+1}q^{-d(j_2+1)}$. As both $\gamma_{j_1+1}$ and $\lambda_{j_2+1}$ are less than $q^d$, $I_{j_1,j_2}$ is the disjoint union of at most $q^{2d}$ boxes of the form \begin{equation}I_1\times I_2:=\left[\frac{a}{q^{d(j_1+1)}}, \frac{a+1}{q^{d(j_1+1)}}\right)\times\left[\frac{b}{q^{d(j_2+1)}}, \frac{b+1}{q^{d(j_2+1)}}\right)\label{interval}\end{equation} for some $a<q^{d(j_1+1)}$ and $b<q^{d(j_2+1)}$. The goal is to count how many entries of $(\mathbf{H}_n(\Phi,P))_{n< N}$ fall inside $I_1\times I_2$. \\
    
    \noindent To this end, expand $N$ as \[N=\sum_{i=0}^{M} N_iq^{d i}\] where $N_i<q^d$. For $i>M$, define $N_i=0$. Similarly, expand $$\Theta(t)=\sum_{i=1}^\infty a_i t^{-i}, ~~~\text{and hence by definition,}~~~\Phi(t)=\sum_{i=1}^\infty a_i P(t)^{-i}$$ with $a_i\in\F_q$. Let $n$ be a natural number and recall Definition \ref{ass} to find the polynomial $n(t)\in\F_q[t]$ associated to $n$. Expand this polynomial in base $P(t)$ as \[n(t)=\sum_{i=0}^\infty n_i(t) P(t)^i,\] where $n_i(t)\in\F_q[t]$ satisfies $\deg(n_i(t))<d$ for all $i$ and $n_i(t)=0$ for all $i>\left\lfloor\frac{\deg(n(t))}{d}\right\rfloor$. Then, by Lemma \ref{kronmat}, $n$ being such that $K_n(\Theta)\in I_1$ is equivalent to finding solutions to the following Hankel matrix equation: \begin{equation}
        \left(\begin{array}{ccccccc}
            a_1 & \dots & a_{j_2+1} & a_{j_2+2} & \dots & a_{M+1} & \dots\\
            \vdots & & \vdots & \vdots & & \vdots & \\
            a_{j_1+1} & \dots & a_{j_1+j_2+1} & a_{j_1+j_2+2}
            & \dots & a_{M+j_1} & \dots
        \end{array}\right)\left(\begin{array}{c}
            n_0(t)\\\vdots\\n_{j_2}(t) \\n_{j_2+1}(t)\\\vdots\\n_{M}(t)\\ \vdots
        \end{array}\right)=\mathbf{z}.
    \label{matrix}\end{equation}
    \noindent Above, $\mathbf{z}$ is some fixed $(j_1+1)$-dimensional vector whose entries are polynomials in $\F_q[t]$ with degree less than $d$. The precise value of $\textbf{z}$ depends on $a$ from (\ref{interval}). Such a vector $\mathbf{z}$ exists because $n<N$ is bounded, so only solutions satisfying $n_i(t)=0$ for $i>M$ are of interest. Additionally, $n<N$ means one does not have complete freedom when choosing the values of $n_i$. That is, counting the solutions to (\ref{matrix}) naively will include some solution vectors that correspond to $n>N$. Therefore, a more precise counting strategy is implemented. \\
    
    \noindent The case where $I_{j_1,j_2}\subset S_1$ is dealt with first. Let \begin{equation}\label{irange}0\le i\le M-j_1-j_2-2-D(\Theta(t)).\end{equation} This range is non-empty by the assumption that $I_1\times I_2\subset I_{j_1,j_2}\subset S_1$. Define \[\mathcal{N}_i:=\sum_{c=0}^{i-1} N_{M-c}q^{d(M-c)}.\] Additionally, $\mathcal{N}_0:=0$. Let $K<N_{M-i}$ be a natural number and assume $n$ is in the range\begin{equation}\mathcal{N}_i+Kq^{d(M-i)}\le n <\mathcal{N}_i+(K+1)q^{d(M-i)}.
    \label{nrange}\end{equation} Then (\ref{matrix}) is rewritten as \begin{align}\nonumber
        \overbrace{\left(\begin{array}{ccc|ccc}
            a_1 & \dots & a_{j_2+1}& a_{M-i+1} & \dots & a_{M+1}\\
            \vdots & &\vdots&\vdots&&\vdots\\
            a_{j_1+1} & \dots & a_{j_1+j_2+1} & a_{M+j_1-i+1}&\dots& a_{M+j_1+1}
        \end{array}\right)\left(\begin{array}{c}
             n_0(t)\\\vdots\\n_{j_2}(t)\\\hline n_{M-i}(t)\\\vdots\\n_{M}(t)
        \end{array}\right)}^{:=\zeta}\\+\underbrace{\left(\begin{array}{ccc}
            a_{j_2+2} & \dots & a_{M-i} \\
            \vdots & &\vdots\\
            a_{j_2+j_1+2} & \dots & a_{M+j_1-i}
        \end{array}\right)}_{:=\Xi}\cdot\underbrace{\left(\begin{array}{c}
             n_{j_2+1}(t)\\\vdots\\n_{M-i-1}(t)
        \end{array}\right)}_{:=\textbf{x}}=  \mathbf{z}
    \label{matrix2}\end{align}
    \noindent Recall the Laurent Series Evaluation Function (\ref{polyeval}). By the range (\ref{nrange}), $E_q(n_c(t))=N_{c}$ for $M-i+1\le c \le M$ and $E_q(n_{M-i}(t))=K$. Furthermore, by Lemma \ref{mod}, $V_n(P)\in I_2$ implies the value of $n_c(t)$ is fixed for $0\le c \le j_2$. Therefore, the part of (\ref{matrix2}) labelled $\zeta$ is fixed. \\
    
    \noindent The difference between the number of columns and rows in the matrix $\Xi$ is $M-i-j_2-j_1-2$. By assumption that $I_1\times I_2$ is in $S_1$, \[M-i-j_1-j_2-2\ge D(\Theta)\] and hence by Lemma \ref{rank} the matrix $\Xi$ has full rank. Therefore, there are exactly $q^{d(M-i-j_1-j_2-2)}$ solutions to (\ref{matrix2}). Completing the same calculation $N_i$ times for all values of $i$ in range (\ref{irange}) shows there are at least \[\sum_{c=M}^{D(\Theta)+j_1+j_2+2}N_cq^{d(c-j_1-j_2-2)}\] values of $n<N$ such that $\mathbf{H}_n\in I_1\times I_2$. However, every value of $n$ in the range $\mathcal{N}_{M-j_1-j_2-1-D(\Theta)}\le n \le N$ is yet to be considered. That is, there are \[\sum_{c=0}^{D(\Theta)+j_1+j_2+1}N_{c}q^{d c}\] values of $n$ that are unaccounted for. Here, an over estimate is achieved by a naive count. \\
    
    \noindent In this case, $i=M-j_1-j_2-1-D(\Theta)$ and hence the matrix $\Xi$ has $j_1+D(\Theta(t))$ columns and $j_1+1$ rows. The submatrix comprised of the first $j_1$ rows has full rank, and therefore the nullity is at most $D(\Theta(t))$. \\
    
    \noindent The vector $\textbf{x}$ is split up as $\textbf{x}=\sum_{j=0}^{d-1}\textbf{x}_{j}t^j$, and similarly let $\textbf{z}=\sum_{j=0}^{d-1}\textbf{z}_{j}t^j$. Hence, the equation $\Xi \textbf{x}_{j}=\textbf{z}_j$ has at most $q^{D(\Theta)}$ solutions, which implies there are at most $q^{d D(\Theta(t))}$ solutions to (\ref{matrix2}). Once again, some of these solutions correspond to $n>N$. But as an over estimate, this suffices. \\
    
    \noindent Recall Definition \ref{discrepancy} and that $\mu_2$ denotes the two dimensional Lebesgue measure. Hence, \begin{align}|\#(I_1\times I_2,\mathbf{H}(\Phi,P),N)-N\mu_2(I_1\times I_2)|&=\left|\sum_{i=0}^{M-j_1-j_2-D(\Theta(t))-2}N_{M-i}(t)q^{d(M-j_1-j_2-2-i)}+q^{d D(\Theta(t))}-Nq^{-d(j_1+j_2+2)}\right|\nonumber\\&=q^{d D(\Theta(t))}.\nonumber\end{align}
    Applying the triangle inequality yields \[|\#(I_{j_1,j_2},\mathbf{H}(\Phi,P),N)-N\mu_2(I_{j_1,j_2})|\le  q^{d D(\Theta(t))+2d}.\]
    
    \noindent Now assume $I_{j_1,j_2}$ is inside $S_2$. As before, $I_{j_1,j_2}$ is the disjoint union of at most $q^{2d}$ intervals of the form $I_1\times I_2$, as seen in (\ref{interval}). By assumption, the following inequality is satisfied: \begin{equation}j_1+j_2+2>M-D(\Theta(t)).\label{S2}\end{equation} By the definition of $M$, this is equivalent to $$j_1+j_2+2\ge\log_{q^d}(N)-D(\Theta(t)).$$ Hence, \[N\mu_2(I_1\times I_2)=\frac{N}{q^{d(j_1+j_2+2)}}\le \frac{q^{d D(\Theta(t))}N}{q^{d\log_{q^d}(N)}}=q^{d D(\Theta(t))}.\] One is now let to calculate $\#(I_1\times I_2,\mathbf{H}(\Phi,P),N)$. Once again, an over estimate produced by a naive count will suffice. As there is no further restriction on $n<N$, the Hankel matrix $\Xi$ is now equal to $H_A(j_2+2,j_1,M-j_2-1)$. This matrix has $j_1+1$ rows and by (\ref{S2}) it has less than $j_1+D(\Theta(t))+1$ columns. Say the number of columns is $j_1+D(\Theta(t))+1-k$ for some natural number $k$. Then the first $j_1+1-k$ rows have full rank by the assumption that $\Theta(t)\in \textbf{Bad}(P(t),q)$. Therefore, the rank is always greater than $j_1+1-k$, meaning the nullity is less than or equal to $D(\Theta(t))$ and so there are at most $q^{d D(\Theta(t))}$ solutions to (\ref{matrix2}). Hence, $\#(I_1\times I_2,\mathbf{H}(\Phi,P),N)\le q^{d D(\Theta(t))}$ and so \[|\#(I_1\times I_2,\mathbf{H}(\Phi,P),N)-N\mu_2(I_1\times I_2)|\le q^{d D(\Theta(t))}\] as needed.\\

    \noindent Note that only the size of the intervals $I_1$ and $I_2$ are important in the case $I_{j_1,j_2}\subset S_2$. Hence, in the final case where the interval is in $S_3$, $S_4$ or $S_5$, the same proof as the $S_2$ case holds since these intervals are simply translations of $I_{j_1,j_2}$ for maximal values of $j_1$ or $j_2$.\\
    
    \noindent To conclude the proof, the interval $S$ has been split into $O(\log(N)^2)$ intervals $I_{j_1,j_2}$, each of which satisfies \[|\#(I_{j_1,j_2},\mathbf{H}(\Phi,P),N)-N\mu_2(I_{j_1,j_2})|\ll_{\Theta,q}1.\] Therefore, by the triangle inequality, $D^*_{N}(\mathbf{H})\ll_{\Theta(t),P(t)} \log^2(N)/N$.
\end{proof}
\section{Open Problems and Conjectures}\label{sect:conj}
\noindent The following conjecture is a natural continuation of the work presented in this paper: \begin{conjecture}\label{all}
    Let $P(t)\in\F_q[t]$ be an irreducible polynomial. Then the hybrid sequence $\left(K_n(\Theta), V_n(P(t))\right)$ generated from some $\Theta(t)\in\textbf{Bad}(P(t),q)$ is low discrepancy.
\end{conjecture} \noindent This conjecture is a stronger version of Theorem \ref{main}, as it does not require that $\Theta(t)\in\textbf{Bad}(P(t),q)$ be induced from $\textbf{Bad}(t,q)$ by Theorem \ref{induce}. Currently, no such $\Theta(t)$ are known to exist.\\

\noindent Additionally, the problem of finding low discrepancy Kronecker-Halton sequences in dimension $d\ge3$ is still open. This motivates the following generalised version of Conjecture \ref{all}:
\begin{conjecture}\label{all2}
    Let $\Theta(t)\in\F_q((t^{-1}))$ be a Laurent series, let $k$ be a natural number and let $P_1(t)$, $\dots$, $P_k(t) \in\F_q[t]$ be coprime irreducible polynomials. Assume that $\Theta(t)\in\mathbf{Bad}(P_i(t),q)$ for all $1\le i \le k$. Then the $k+1$ dimensional digital Kronecker-Halton sequence defined by $\Theta(t)$ and the polynomials $P_i(t)$ is low discrepancy. 
\end{conjecture}
\noindent In \cite[Theorm 2]{bad}, Hofer proved the equivalent statement for the weaker condition $\Theta(t)\in \textbf{Bad}(q)$, getting a discrepancy bound of $$D_{N,\mathbf{H}}\ll \frac{\log^{k+1}(N)}{\sqrt{N}}\cdotp$$ Even if one were to prove Conjecture \ref{all2}, no Laurent series are known to be in $\textbf{Bad}(P(t),q)$ for multiple different polynomials $P(t)$.\\

\noindent It is also reasonable to inquire about a lower bound for the discrepancy of the Hybrid sequences from Theorem \ref{main} to verify that they do not beat the conjectured best possible \ref{conj}. Indeed, perhaps recent methods of Levin \cite{Levin5,Levin1} could be applied to prove the following conjecture: \begin{conjecture}
    The hybrid sequence $\left(\textbf{H}_n\right)_{n\ge0}$ from Theorem \ref{main} satisfies \[D_{N,\textbf{H}} \gg \frac{\log^2(N)}{N}\cdotp\] \end{conjecture}

\noindent Finally, the real version of Theorem \ref{main} is still open. Recall the set $\textbf{Bad}_p$ from Definition \ref{badp}. \begin{conjecture}
    If $\alpha\in\R$ is in $\textbf{Bad}_p$ for some prime $p$, then the hybrid sequence comprised of the Kronecker sequence associated to $\alpha$ and the Van der Corput sequence associated to $p$ is low discrepancy.\label{fin}
\end{conjecture}
\noindent The value of the above conjecture is not in the possibility of finding new low discrepancy hybrid sequences, as it is expected that $\textbf{Bad}_p$ is empty. Instead, Conjecture \ref{fin} would allow one to test the obstructions for $p$-$LC$ to be false by showing that existence of a counterexample would imply strong constraints on the existence low discrepancy sequences. 
\printbibliography
\end{document}